\documentclass[11pt]{article}

\usepackage{amsfonts,amsmath,latexsym,color,epsfig}
\newtheorem{theorem}{Theorem}
\newtheorem{lemma}{Lemma}

\newtheorem{corollary}{Corollary}

\newtheorem{claim}{Claim}
\newtheorem{problem}{Problem}

\newenvironment{proof}
      {\medskip\noindent{\bf Proof:}\hspace{1mm}}
      {\hfill$\Box$\medskip}

\input{epsf}

\makeatletter
\def\Ddots{\mathinner{\mkern1mu\raise\p@
\vbox{\kern7\p@\hbox{.}}\mkern2mu
\raise4\p@\hbox{.}\mkern2mu\raise7\p@\hbox{.}\mkern1mu}}
\makeatother

\title{\vspace{-0.7cm}Ramsey numbers of sparse hypergraphs}
\author{David Conlon\thanks{St John's College, Cambridge, United Kingdom. E-mail: {\tt
D.Conlon@dpmms.cam.ac.uk}} \and Jacob Fox\thanks{Department of
Mathematics, Princeton, Princeton, NJ. Email: {\tt
jacobfox@math.princeton.edu}. Research supported by an NSF Graduate
Research Fellowship and a Princeton Centennial Fellowship.} \and
Benny Sudakov\thanks{Department of Mathematics,
UCLA,  Los Angeles, CA 90095 and Institute for Advanced Study, Princeton, NJ. Email: {\tt
bsudakov@math.ucla.edu}.
Research supported in part by NSF CAREER award DMS-0546523, NSF
grants DMS-0355497 and DMS-0635607, by a USA-Israeli BSF grant, and by the State of New
Jersey.}}
\date{}
\begin{document}
\maketitle

\begin{abstract}
We give a short proof that any
$k$-uniform hypergraph $H$ on $n$ vertices with bounded degree
$\Delta$ has Ramsey number at most $c(\Delta, k) n$, for an
appropriate constant $c(\Delta, k)$. This result was recently proved by several authors,
but those proofs are all based on
applications of the hypergraph regularity method. Here we give a
much simpler, self-contained proof which uses new techniques developed recently by the
authors together with an argument of Kostochka and R\"odl. Moreover, our
method demonstrates that, for $k \geq 4$,
\[c(\Delta, k) \leq 2^{2^{\Ddots^{2^{c \Delta}}}},\]
where the tower is of height $k$ and the constant $c$ depends on
$k$. It significantly improves on the Ackermann-type upper bound
that arises from the regularity proofs, and we present a
construction which shows that, at least in certain cases, this
bound is not far from best possible. Our methods also allows us to prove quite sharp results on the Ramsey number of
hypergraphs with at most $m$ edges.
\end{abstract}

\section{Introduction}

For a graph $H$, the {\it Ramsey number} $r(H)$ is the least
positive integer $N$ such that, in every two-colouring of the edges
of complete graph $K_N$ on $N$ vertices, there is a monochromatic
copy of $H$. Ramsey's theorem states that $r(H)$ exists for every
graph $H$. A classical result of Erd\H{o}s and Szekeres, which is a
quantitative version of Ramsey's theorem, implies that $r(K_k) \leq
2^{2k}$ for every positive integer $k$. Erd\H{o}s showed using
probabilistic arguments that $r(K_k) > 2^{k/2}$ for $k
> 2$. Over the last sixty years, there have been several
improvements on these bounds (see, e.g., \cite{Co}). However,
despite efforts by various researchers, the constant factors in the
above exponents remain the same.

Determining or estimating Ramsey numbers is one of the central
problem in combinatorics, see the book {\it Ramsey theory}
\cite{GRS80} for details. Besides the complete graph, the next most
classical topic in this area concerns the Ramsey numbers of sparse
graphs, i.e., graphs with certain upper bound constraints on the
degrees of the vertices. The study of these Ramsey numbers was
initiated by Burr and Erd\H{o}s in 1975, and this topic has since
placed a  central role in graph Ramsey theory. Burr and Erd\H{o}s
conjectured, and it was proved by Chv\'atal, R\"odl, Szemer\'edi and
Trotter \cite{CRST83}, that for every graph $G$ on $n$ vertices and
maximum degree $\Delta$,
\[r(G) \leq c(\Delta) n.\]
Their proof of this theorem is a classic application of Szemer\'edi's beautiful
regularity lemma. However, the use of this lemma makes the upper bound on $c(\Delta)$ grow as a tower of $2$s with height
proportional to $\Delta$. Eaton \cite{E98} used a variant of the regularity lemma
to obtain the upper bound $c(\Delta) \leq 2^{2^{c \Delta}}$ for some fixed $c$. A novel approach of Graham,
R\"odl, Rucinski \cite{GRR00} that did not use any form of the regularity lemma gives the
upper bound $c(\Delta) \leq 2^{c\Delta \log^2 \Delta}$ for some fixed $c$. In the other direction, in \cite{GRR01}
they proved that there is a positive constant $c$ such that, for every
$\Delta \geq 2$ and $n \geq \Delta+1$, there is a bipartite graph
$G$ with $n$ vertices and maximum degree at most $\Delta$ satisfying $r(G)
\geq 2^{c\Delta}n$. Recently, the authors \cite{C07}, \cite{FS07}
closed the gap for bipartite graphs by showing that, for every bipartite graph $G$ with $n$ vertices and maximum degree
$\Delta$, $r(G) \leq 2^{c\Delta}n$ for some fixed $c$.

A {\it hypergraph} $H=(V,E)$ consists of a vertex set $V$ and an edge set
$E$, which is a collection of subsets of $V$. A hypergraph is
{\it $k$-uniform} if each edge has exactly $k$ vertices. The {\it Ramsey number}
$r(H)$ of a $k$-uniform hypergraph $H$ is the smallest number $N$
such that, in any $2$-colouring of the edges of the complete
$k$-uniform hypergraph $K_N^{(k)}$, there is guaranteed to be a
monochromatic copy of $H$. The existence of these numbers was proven by Ramsey \cite{R30},
but no proper consideration of the values of
these numbers was made until the fifties, when Erd\H{o}s and Rado
\cite{ER52}. To understand the growth of Ramsey numbers for hypergraphs,
it is useful to introduce the tower function $t_i(x)$, which is defined
by $t_1(x)=x$ and $t_{i+1}(x)=2^{t_i(x)}$, i.e.,
\[t_{i+1}(x)=2^{2^{\Ddots^{2^{x}}}},\]
where the number of $2$s in the tower is $i$. Erd\H{o}s and Rado showed that for $H$ being the complete $k$-uniform hypergraph $K_l^{(k)}$,
$r(H) \leq t_k(cl)$, where the constant $c$ depends on $k$. In the other direction, Erd\H{o}s and Hajnal (see \cite{GRS80})
proved that for
$H = K_l^{(k)}$,
$r(H) \geq t_{k-1}(cl^2)$, where the constant $c$ depends on $k$.

One can naturally try to extend the sparse graph Ramsey results to hypergraphs. Kostochka and R\"odl \cite{KR06}
showed that for every $\epsilon>0$, the Ramsey number of any $k$-uniform hypergraph $H$
with $n$ vertices and maximum degree $\Delta$ satisfies \[r(H) \leq c(\Delta, k, \epsilon) n^{1 + \epsilon},\]
where $c(\Delta,k,\epsilon)$ only depends on $\Delta$, $k$, and $\epsilon$. Since the first proof of the sparse graph Ramsey
theorem used Szemer\'edi's regularity lemma, it was therefore natural to expect that, given the
recent advances in developing a hypergraph regularity method
\cite{G07, RS04, NRS06}, linear bounds might as well be provable
for hypergraphs. Such a program was indeed recently pursued by
several authors (Cooley, Fountoulakis, K\"uhn, and Osthus \cite{CFKO07, CFKO072}; Nagle, Olsen, R\"odl, and Schacht
\cite{NORS07}; Ishigami \cite{I07}), with the result that we now
have the following theorem:

\begin{theorem}
Let $\Delta$ and $k$ be positive integers. Then there exists a
constant $c(\Delta,k)$ such that the Ramsey number of any
$k$-uniform hypergraph $H$ with $n$ vertices and maximum degree
$\Delta$ satisfies
\[r(H) \leq c(\Delta, k) n.\]
\end{theorem}

In this paper we will give a short proof of this theorem, which is much simpler and avoids all use of the regularity lemma,
building instead on techniques developed recently by Conlon
\cite{C07} and by Fox and Sudakov \cite{FS07} in order to study embeddings of sparse bipartite graphs in dense graphs.

The first main result of this paper is an extension of this work
from graphs to hypergraphs. An $l$-uniform hypergraph is {\it $l$-partite} if there is a partition of the vertex set into
$l$ parts such that each edge has exactly one vertex in each part. We prove the following Tur\'an-type
result for $l$-uniform $l$-partite hypergraphs:

\begin{theorem} There exists a constant $c=c(l)$ such that if $F$ is
an $l$-uniform $l$-partite hypergraph with $n$ vertices and
maximum degree $\Delta$ and $G$ is an $l$-uniform $l$-partite hypergraph with parts of size $N \geq
\left(\epsilon/2\right)^{-c \Delta^{l-1}} n$ and at least $\epsilon N^l$ edges, then $G$ contains a copy of $F$.
\end{theorem}

Then, in section 3, we will prove Theorem 1 by applying an argument
of Kostochka and R\"odl which shows that the Ramsey problem for general
hypergraphs may be reduced to an application of the Tur\'an theorem
in the $l$-uniform $l$-partite case. This argument combined with our Theorem 2 shows that,
for $k \geq 4$ and $k$-uniform hypergraph $H$ with $n$ vertices and maximum degree $\Delta$,
\[r(H) \leq t_k(c\Delta)n,\]
where the constant $c$ depends on
$k$. For $k=3$, the proof shows that $r(H) \leq t_3(c\Delta\log \Delta)n$.
This is clearly much better than the Ackermann-type upper bound that
arises from the regularity proofs. The tower-type upper bound cannot be avoided as demonstrated by
the lower bound of Erd\H{o}s and Hajnal for the Ramsey number of the
complete $k$-uniform hypergraph on $n$ vertices. This hypergraph has
maximum degree $\Delta={n-1 \choose k-1}$ and Ramsey number at least $t_{k-1}(c\Delta^{\frac{2}{k-1}})n$, where
the constant $c$ depends on $k$.

For $k$-uniform hypergraphs $H_1,\ldots,H_q$, the {\it multicolour Ramsey
number} $r(H_1,\ldots,H_q)$ is the minimum $N$ such that, in any
$q$-colouring of the edges of the complete $k$-uniform hypergraph
$K_N^{(k)}$ with colours $1,\ldots,q$, there is a monochromatic copy
of $H_i$ in colour $i$ for some $i$, $1 \leq i \leq q$. The proof of Theorem 1
presented here extends in a straightforward manner to the
multicolour generalisation, which states that for all positive integers $\Delta$, $k$, and $q$, there exists a
constant $c(\Delta,k,q)$ such that, if $H_1,\ldots,H_q$ are
$k$-uniform hypergraphs each with $n$ vertices and maximum degree
$\Delta$, then $r(H_1,\ldots,H_q) \leq c(\Delta, k,q) n$. The proof demonstrates
that may take $c(\Delta,k,q) \leq t_k(c\Delta)$ for $k \geq 4$
and $c(\Delta,3,q) \leq t_3(c \Delta \log \Delta)$, where the constant $c$ depends on $k$ and $q$.
In the other direction,
in Section 4 we construct, for each sufficiently large $\Delta$,  a $3$-uniform hypergraph $H$ with maximum degree at most $\Delta$
for which the $4$-colour Ramsey number of $H$ satisfies $r(H,H,H,H) \geq t_3(c\Delta)n$,
where $n$ is the number of vertices of $H$. This example shows that our upper bound for hypergraph Ramsey numbers
is probably close to being best possible.

The same example also shows that there is a $3$-uniform hypergraph $H$ with $m$ edges for which the $4$-colour
Ramsey number of $H$ is at least $t_3(c\sqrt{m})$. On the other hand, one can easily deduced from the proof of Theorem 1 that for any $k$-uniform hypegraph $H$ with
 $m$ edges, we have that the $q$-colour Ramsey number of $H$ satisfies $r(H,\cdots,H) \leq t_3(c\sqrt{m}\log m)$ for $k=3$,
 and $r(H,\cdots,H) \leq t_k(c\sqrt{m})$ for $k \geq 4$, where $c$ depends on $k$ and $q$.

\section{A Tur\'an theorem for $l$-uniform $l$-partite hypergraphs}

The following is a generalisation to hypergraphs of a lemma which
has appeared increasingly in the literature on Ramsey theory, whose
proof uses a probabilistic argument known as {\it dependent random
choice}. Early versions of this technique were developed in the
papers \cite{G98}, \cite{KR01}, \cite{S03}. Later, variants were
discovered and applied to various Ramsey and density-type problems
(see, e.g., \cite{KS03,AKS03,S05,KR06,FS07,C07}). We define the {\it weight} $w(S)$ of a set $S$ of edges in a hypergraph to be the size of the union of these edges.
\begin{lemma}
Suppose $s,\Delta$ are positive integers, $\epsilon,\beta > 0$, and $G_r = (V_1, \cdots, V_r; E)$ is an
$r$-uniform $r$-partite hypergraph with $|V_1| = |V_2| = \cdots =
|V_r| = N$ and at least $\epsilon N^r$ edges. Then there exists an $(r-1)$-uniform $(r-1)$-partite hypergraph $G_{r-1}$
on the vertex sets $V_2, \cdots, V_r$ which has at least
$\frac{\epsilon^s}{2} N^{r-1}$ edges and such that for each nonnegative integer $w \leq (r-1)\Delta$,
there are at most $4r\Delta\epsilon^{-s}\beta^{s}w^{r\Delta}r^wN^w$ dangerous sets of edges of $G_{r-1}$ with weight $w$,
where a set $S$ of edges of $G_{r-1}$ is dangerous if $|S| \leq \Delta$ and the number of vertices $v \in V_1$ such that
for every edge $e \in S$, $e+v \in G_r$ is less than $\beta N$.
\end{lemma}

\begin{proof}
Let $\cal{C}$ be the complete $(r-1)$-uniform $(r-1)$-partite
hypergraph on the vertex sets $V_2, \cdots, V_r$. For any edge $e$
in $\cal{C}$, let $d(e)$ be the degree of $e$ in $G_r$, i.e., the
number of vertices in $V_1$ such that $e + v \in G_r$. Let $T$ be
a set of $s$ random vertices of $V_1$, chosen uniformly with
repetitions. Let $A$ be the set of edges in $\cal{C}$ which are
common neighbours of the vertices of $T$, i.e., an edge $e$ of
$\cal{C}$ is in $A$ if $e+v$ is an edge of $G_r$ for all $v \in
T$. Let $X$ denote the cardinality of $A$. We will show that with
positive probability, the set $A$ will be the edge set of a
hypergraph $G_{r-1}$ on vertex sets $V_2,\ldots,V_r$ with the
desired properties. By linearity of expectation and by convexity
of $f(z)=z^s$,
\begin{eqnarray*}
\mathbb{E}[X] & = & \sum_{e \in \cal{C}} \mathbb{P}[e \in A] =
\sum_{e \in \cal{C}} \left( \frac{d(e)}{N} \right)^s \\ & \geq &
\frac{N^{r-1} \left(\frac{\sum_{e \in \cal{C}} d(e)}{N^{r-1}}
\right)^s}{N^s} \geq \frac{N^{r-1} (\epsilon N)^s}{N^s} = \epsilon^s
N^{r-1}.
\end{eqnarray*}

Note that $X \leq N^{r-1}$ since $\cal{C}$ has $N^{r-1}$ edges. Letting $p$ denote the probability that $X \geq \mathbb{E}[X]/2$,
we have $$\mathbb{E}[X] \leq (1-p)\mathbb{E}[X]/2+pN^{r-1} \leq \mathbb{E}[X]/2+pN^{r-1}.$$
So the probability $p$ that $X \geq \mathbb{E}[X]/2 \geq \epsilon^sN^{r-1}/2$
satisfies $p \geq \frac{\mathbb{E}[X]}{2N^{r-1}} \geq \epsilon^s/2$.

The number of subsets $S$ of $V_2 \cup \ldots \cup V_r$ of size $w$ is ${(r-1)N \choose w}$.
For a given $w$-set $S$, the number of collections $\{e_1,\ldots,e_t\}$ of size $t$ with $|e_i|=r-1$,
and $e_i \subset S$ for $1 \leq i \leq t$
is ${{w \choose r-1} \choose t}$. Hence, summing over all nonnegative $t \leq \Delta$,
the number of sets of edges of $\cal{C}$
 with weight $w$ and size at most $\Delta$ is at most
\[\sum_{t=0}^{\Delta} {{w \choose r-1} \choose t}{(r-1)N \choose w} \leq w^{r\Delta}(rN)^w=w^{r\Delta}r^wN^w,\]

Let $Y_w$ denote the random variable counting the number of dangerous sets $S$ of edges of $G_{r-1}$ with weight $w$.
We next give an upper bound on $\mathbb{E}[Y_w]$. For a given set $S$ of edges of $\cal{C}$, the probability
that $S$ is a subset of edges of $G_{r-1}$ is $\left(\frac{|N(S)|}{N} \right)^s$, where $N(S)$ denotes the
set of vertices $v \in V_1$ with $v+e$ an edge of $G_{r}$ for all $e \in S$. So if $S$ satisfies $N(S)<\beta N$, then
the probability that $S$ is a subset of edges of $G_{r-1}$ is less than $\beta^s$. By linearity of expectation, we have
$\mathbb{E}[Y_w] < \beta^sw^{r\Delta}r^wN^w $.

Let $\alpha=4r\Delta\epsilon^{-s}$. Since $Y_w$ is a nonnegative random variable,
Markov's inequality implies that $\mathbb{P}\left(Y_w \geq \alpha \mathbb{E}[Y_w]\right) \leq \frac{1}{\alpha}$.
Hence, the probability that there is a nonnegative integer $w \leq (r-1)\Delta$ with
$Y_w \geq \alpha \beta^{s} w^{r\Delta}r^wN^w$
is at most $r\Delta/\alpha=\epsilon^s/4$.
Since the probability that $X \geq
\frac{\epsilon^s}{2} N^{r-1}$ is at least $\epsilon^s/2$, we can satisfy
the conditions of the lemma with probability at least $\epsilon^s/4$.
\end{proof}

By simply iterating the previous lemma $l-1$ times, we obtain the following corollary.

\begin{corollary}
Suppose $s,\Delta$ are positive integers, $\epsilon,\beta > 0$,  and $G_l = (V_1, \cdots, V_l; E_l)$ is an
$l$-uniform $l$-partite hypergraph with $|V_1| = |V_2| = \cdots =
|V_l| = N$ and at least $\epsilon N^l$ edges. Let $\delta_l=\epsilon$ and $\delta_{r-1}=\delta_{r}^s / 2$
for $2 \leq r \leq l$. Then, for $1 \leq r \leq l-1$, there are $r$-uniform $r$-partite hypergraphs
$G_{r}=(V_{l-r+1},\ldots,V_l,E_r)$
with the following properties:
\begin{enumerate}
\item $G_{r}$ has at least $\delta_{r}N^{r}$ edges for $1 \leq r \leq l$, and
\item for $2 \leq r \leq l$ and each nonnegative integer $w \leq (r-1)\Delta$,
there are at most $4r\Delta\delta_{r}^{-s}\beta^{s}w^{r\Delta}r^wN^w$ dangerous sets of $G_{r-1}$
with weight $w$, where a set $S$ of edges of $G_{r-1}$ is dangerous if $|S| \leq \Delta$
and the number of vertices $v \in V_{l-r+1}$ such that
for every edge $e \in S$, $e+v \in G_r$ is less than $\beta N$.
\end{enumerate}
\end{corollary}

This is all the preparation we need before proving our main
contribution, Theorem 2. For the proof, we will use Corollary 1 and then show how to embed $F$ into $G$. The
latter part is closely related to the many embedding results
proven by Fox and Sudakov in \cite{FS07}. We will actually prove the following more precise version of
Theorem 2:

\begin{theorem}
Let $l \geq 3$, $F$ be an $l$-uniform $l$-partite hypergraph, on vertex sets
$W_1, \cdots, W_l$, with at most $n$ vertices and maximum degree
$\Delta$. Let $G_l$ be an $l$-uniform $l$-partite graph, on vertex
sets $V_1, \cdots, V_l$ with $|V_1| = \cdots = |V_l| = N$, with at
least $\epsilon N^l$ edges. Then, provided that $N \geq
\left(\epsilon/2\right)^{-(2l\Delta)^{l-1}} n$,
$G_l$ contains a copy of $F$.
\end{theorem}

\begin{proof}
We apply Corollary 1 with $s=2l\Delta$, $\delta_l=\epsilon$, $\delta_{i-1}=\delta_{i}^s/2$ for $2 \leq i \leq l$,
and $\beta=2\left(\epsilon/2\right)^{(2l\Delta)^{l-1}}$ to get hypergraphs $G_{l-1},\ldots,G_1$.
It is easy to check by induction on $i$ that $\delta_{l-i} = 2^{-(s^i-1)/(s-1)}\epsilon^{s^i}$,
so $$\delta_1 = 2^{-(s^{l-1}-1)/(s-1)}\epsilon^{s^{l-1}} \geq 2\left(\epsilon/2\right)^{(2l\Delta)^{l-1}}=\beta$$ and
$\delta_1N \geq \beta N \geq 2n$.

We now construct an $l\Delta$-uniform {\it bad hypergraph} $B$ with
vertex set $V_1 \cup \ldots \cup V_l$ where each edge of $B$ has
exactly $\Delta$ vertices in each $V_i$. A set $T \subset V_1 \cup
\ldots \cup V_l$ which contains exactly $\Delta$ vertices in each
$V_i$ is an edge of $B$ if and only if there is a dangerous set $S$
of edges of $G_r$ for some $r$, $1 \leq r \leq l-1$, such that the
union of the edges of $S$ is a subset of $T$. In other words, an
edge of $B$ is just an extension of the union of the edges of a
dangerous set. For a particular dangerous set $S$ of edges with
weight $w$ in some $G_r$, the number of edges of $B$ that are
extensions of the union of the edges in $S$ is at most
$N^{l\Delta-w}$ since there are at most $N$ ways to pick each of the
$l\Delta-w$ remaining vertices that make up an edge. Summing over
all $r$ and $w$, and using the fact that $l \geq 3$ and
$\delta_2^s=2\delta_1 \geq 2\beta$, the number of edges of $B$ is at
most
\begin{eqnarray*}
\sum_{r=2}^{l} \sum_{w=0}^{(l-1)\Delta} 4r\Delta\delta_{r}^{-s}\beta^{s}w^{r\Delta}r^wN^{l\Delta} & \leq & l^2\Delta \cdot
2l\Delta \beta^{s-1}(l\Delta)^{l\Delta}l^{l\Delta}N^{l\Delta} = 2l^3 \Delta^2 (l^2 \Delta)^{l\Delta} \beta^{l \Delta-1} \beta^{l \Delta} N^{l\Delta}
\\ & \leq & 2^{1+3(l \Delta)^2} \beta^{l \Delta-1} \beta^{l \Delta} N^{l\Delta}
\leq  2^{1+3(l \Delta)^2} 2^{(1-(2l\Delta)^{l-1})(l \Delta-1)}
\beta^{l \Delta} N^{l\Delta} \\ & \leq & 2^{-4(l \Delta)^2} \beta^{l
\Delta} N^{l\Delta} <\left(\frac{\beta}{4l\Delta}\right)^{l\Delta}{N
\choose  l\Delta}.
\end{eqnarray*}

Call a set $U \subset V_1 \cup \ldots \cup V_l$ with at most $\Delta$ vertices in each $V_i$ {\it bad}
if there are at least
\[\left(\frac{\beta}{4l\Delta}\right)^{l\Delta-|U|}{N \choose  l\Delta-|U|} \] edges of $B$ that contain $U$;
 otherwise call $U$ {\it good}.
Note that the above calculation on the number of edges of $B$ demonstrates that the empty set is good.
We next prove the following important claim.

\begin{claim} If $S$ is a dangerous set of edges in $G_r$ for some $r$, $1 \leq r \leq l-1$, and $U$ is a good set,
then the union of the edges in $S$ is not a subset of $U$.
\end{claim}
\begin{proof} Suppose for contradiction that the union of the edges in $S$ is a subset of $U$. The number of extensions of $U$ to a set which contains exactly $\Delta$ vertices in each $V_i$ is
\[\prod_{i=1}^l {N-|V_i \cap U| \choose \Delta-|V_i \cap U|}\]
since we can pick for each $i$ any $\Delta-|V_i \cap U|$ vertices of $V_i \setminus U$ to extend $U$. By definition,
all of these sets are edges in $B$. Using the simple fact that if $x_1,\ldots,x_l$ are
nonnegative integers then $\prod_{i=1}^l x_i! \leq (\sum_{i=1}^l x_i)!$, it is straightforward to check that
\begin{eqnarray*}
\prod_{i=1}^l {N-|V_i \cap U| \choose \Delta-|V_i \cap U|} & \geq & (N/2)^{l\Delta-|U|} \prod_{i=1}^l
\left( \Delta-|V_i \cap U|\right)!^{-1} \geq \left(\frac{1}{2}\right)^{l\Delta-|U|}{N \choose  l\Delta-|U|}  \\ & \geq &
\left(\frac{\beta}{4l\Delta}\right)^{l\Delta-|U|}{N \choose  l\Delta-|U|},
\end{eqnarray*}
which contradicts $U$ being good.
\end{proof}

Given a good set $U$ with $|V_i \cap
U|<\Delta$ and $v \in V_i \setminus U$, we say $v$ is {\it bad with respect to $U$} if $U \cup
\{v\}$ is bad. Let $B_U$ denote the set of vertices that are bad
with respect to $U$. We will show that for $U$ good we have $|B_U|
\leq \frac{\beta N}{4 l\Delta}$. Indeed, suppose $|B_U| >
\frac{\beta N}{4l\Delta}$. Then the number of edges of $B$ containing
$U$ is at least
\[\frac{|B_U|}{l\Delta - |U|} \left(\frac{\beta}{4l\Delta}
\right)^{l\Delta - |U| - 1} \binom{N}{l\Delta - |U| - 1} >
\left(\frac{\beta}{4l\Delta}\right)^{l\Delta - |U|}
\binom{N}{l\Delta-|U|},\] contradicting the fact that $U$ is good.

Fix a labeling $\{v_1, \cdots, v_n\}$ of the vertices of $F$ such
that all vertices in $W_{i+1}$ precede all those in $W_i$ for all
$i = 1, \cdots, l-1$. For each $i$, let $L_i = \{v_1, \cdots,
v_i\}$. For each vertex $v_h$, the {\it trace neighbourhood}
$N(v_h)$ is the set of vertices $v_m$ with $m<h$ that are in an
edge of $F$ with $v_m$. Note that $N(v_h)$ contains at most
$\Delta$ vertices in each $W_r$ since $F$ has maximum degree
$\Delta$. We will find an embedding $f$ of the vertices of $F$
such that $f(W_r) \subset V_r$ for $1 \leq r \leq l$ and for each
$i \leq lN$,
\begin{enumerate}
\item $f(N(v) \cap L_i)$ is good for each vertex $v$ of $F$, and
\item $f(e \cap L_i)$ is an edge of $G_{r}$ for each edge $e$ of $F$, where $r=|e \cap L_i|$.
\end{enumerate}
The proof will be complete once we find such an embedding $f$
since, for each edge $e$ of $F$, $f(e \cap L_n)=f(e)$ is an edge
of $G_l$, so $f$ provides an embedding of $F$ in $G_l$. The
embedding will be constructed one vertex at a time, in increasing
order of subscript, so the proof will be by induction on $i$. As
noted earlier, the empty set is good, so our base case $i=0$ is
satisfied.

Suppose then that at step $i$, we have found an embedding $f$ of $v_1,\ldots,v_i$ such that
\begin{enumerate}
\item for each vertex $v$ of $F$, $f(N(v) \cap L_i)$ is good, and
\item for each edge $e$ of $F$, $f(e \cap L_i)$ is an edge of $G_{r}$, where $r=|e \cap L_i|$.
\end{enumerate}

Let $j$ be such that $v_{i+1} \in W_j$. Let $e_1,\ldots,e_d$ denote the edges of $F$ that contain $v_{i+1}$ and
$e_1',\ldots,e_d'$ denote the truncations of $e_1,\ldots,e_d$ by deleting all $j$ vertices from each $e_t$ that are in
some $W_h$ with $h \leq j$. Each $e_t'$ consists of one vertex from each $W_h$ with $h>j$. Also, $d \leq \Delta$ since
$F$ has maximum degree $\Delta$.

Since $F$ has maximum degree $\Delta$, there are less than
$l\Delta$ vertices $v$ for which $v_{i+1} \in N(v)$. For each such $v$,
$f(N(v) \cap L_i)$ is good, so there are at most $\frac{\beta}{4l\Delta}N$ vertices $w$ in $V_j$ for which
$f(N(v) \cap L_i) \cup w$ is bad. Adding over all such $v$, we conclude
that there are at most $\frac{\beta}{4} N$ bad vertices in all
associated with $v_{i+1}$.

Suppose we are still embedding vertices of $W_l$ in $V_l$. Since the edge set of $G_1$
is just a subset of $V_l$ whose size by Corollary 1 is at least $\delta_1N=\beta N$,
 then we can choose any of these at least $\beta N$ vertices other than $f(v_1),\ldots,f(v_i)$ for $f(v_{i+1})$ to satisfy
 the second of the two desired properties for $f(v_{i+1})$. We see that
 there are at least $\beta N-i-\frac{\beta}{4}N > \frac{3\beta N}{4}-n>0$ vertices to choose from
for $f(v_{i+1})$ to satisfy both of the desired properties.

If, now, we have chosen all of the vertices in $W_l, \cdots,
W_{j+1}$ and we are trying to embed vertex $v_{i+1}$ in $W_{j}$ (we
may have already embedded other vertices in $W_j$), we can do so. To
see this, by the induction hypothesis, $f(N(v_{i+1}) \cap
L_{i})=f(N(v_{i+1}))=\bigcup_{t=1}^d f(e_t')$ is good. By Claim 1,
this implies that the set $\{f(e_1'),\ldots,f(e_d')\}$ of edges of
$G_{l-j}$ is not dangerous, i.e., there are at least $\beta N$
vertices $v \in V_j$ such that $f(e_t') \cup v$ is an edge of
$G_{l-j+1}$ for $1 \leq t \leq d$. Therefore, since there are at
most $\frac{\beta}{4} N$ bad vertices associated with $v_{i+1}$ and
we have already chosen $f(v_1),\ldots,f(v_i)$, we have at least
$\frac{3}{4} \beta N - i>\frac{3}{4}\beta N-n>0$ choices for
$f(v_{i+1})$, which completes the proof.
\end{proof}
\section{The Ramsey theorem}

We are now ready to prove Theorem 1 in the following form:

\begin{theorem}
Let $\Delta$ and $k \geq 3$ be positive integers. Then the Ramsey number
of any $k$-uniform hypergraph $H$ with $n$ vertices and maximum
degree $\Delta$ satisfies
\[r(H) \leq r_k(k \Delta)^{(2k\Delta^2)^{k \Delta}} n,\]
where $r_k(l) = r(K_l^{(k)})$.
\end{theorem}

\begin{proof}
We use the argument of Kostochka and R\"odl \cite{KR06} together with Theorem 3.
Let $l = (k-1) \Delta + 1$. Suppose we have a red-blue colouring
of the complete $k$-uniform hypergraph on $N$ vertices. Let $G$ be
the hypergraph consisting of all the red edges and let $r_k(l)$ be
the Ramsey number of the hypergraph $K_l^{(k)}$. Then, in each
subset of the vertices of size $r_k(l)$, there is at least one
monochromatic $K_l^{(k)}$. Counting over all such sets and
dividing out by possible multiple counts we see that we have at
least
\[\frac{\binom{N}{r_k(l)}}{\binom{N-l}{r_k(l) - l}} \geq \frac{N^l}{r_k(l)^l}\]
monochromatic $K_l^{(k)}$. Therefore, either G or its complement
$\overline{G}$ contains at least $N^l/2 r_k(l)^l$ cliques
$K_l^{(k)}$. We will suppose that it is $G$.

Now we pass instead to considering the $l$-uniform hypergraph
$G^{(l)}$, the edges of which are exactly those $l$-tuples which
form complete $K_l^{(k)}$ in $G$. This hypergraph has at least
$N^l/2 r_k(l)^l$ edges. Partition its vertex set randomly into $l$
parts $V_1, \cdots, V_l$ of equal size $N/l$. The total number of
partitions is $\frac{N!}{(N/l)!^l}$ and, for any given edge $e$,
there are
 $l! \frac{(N-l)!}{(N/l-1)!^l}$ partitions such that
each vertex of this edge is in a different part of the partition.
Therefore, the expected number of edges with one vertex in each
set of the random partition is at least
\[e(G^{(l)}) \left(  l! \frac{(N-l)!}{(N/l-1)!^l} \bigg/
\frac{N!}{(N/l)!^l} \right) \geq \frac{N^l}{2 r_k(l)^l}\cdot l! \frac{(N-l)!}{(N/l-1)!^l}
\frac{(N/l)!^l}{N!} \geq \frac{N^l}{2 r_k(l)^l}
\frac{l!}{l^l}= \frac{l!}{2r_k(l)^l}\left(\frac{N}{l} \right)^l.\]
Now choose such a partition and let $\hat{G}^{(l)}$ be the
$l$-uniform $l$-partite subhypergraph of $G^{(l)}$ consisting of
those edges of $G^{(l)}$ which have one edge in each partite set.
Note that $\hat{G}^{(l)}$ has $N/l$ vertices in each part and at least $\epsilon \left(\frac{N}{l}\right)^l$ edges, where
$\epsilon=\frac{l!}{2r_k(l)^l}$.

Now we extend hypergraph $H$ to an $l$-uniform $l$-partite
hypergraph $H^{(l)}$. We first note that the vertices of $H$ can be
partitioned into $l$ subsets $A_1,\ldots,A_l$ such that each edge of
$H$ has at most one vertex in each part. This is equivalent to
saying that the graph $H'$ with the same vertex set as $H$ and with
two vertices adjacent if they lie in an edge of $H$ has chromatic
number at most $l$. Since $H'$ has maximum degree at most
$(k-1)\Delta$, it has chromatic number at most $(k-1)\Delta+1=l$.
For each edge $e$ of $H$, we add one auxiliary vertex to each $A_i$
which is disjoint from $e$ (in total $l-k$ vertices). Note that the
maximum degree of $H^{(l)}$ remains $\Delta$. The total number of
auxiliary vertices added is at most $\frac{\Delta n}{k} \cdot (l-k)
< \Delta(\Delta-1)n$ since there are $l-k$ auxiliary vertices for
each edge and the total number of edges of $H$ is at most
$\frac{\Delta n}{k}$. Hence, $H^{(l)}$ has less than $\Delta^2 n$
vertices.

Applying Theorem 3 with $F=H^{(l)}$, $G_l=\hat{G}^{(l)}$, and $\epsilon=\frac{l!}{2r_k(l)^l}$
we see that,
provided
\[\frac{N}{l} \geq \left(\epsilon/2\right)^{-(2l\Delta)^{l-1}} \cdot \Delta^2 n,\]
$\hat{G}^{(l)}$ contains a copy of $H^{(l)}$. But now, by the
construction of $H^{(l)}$, this implies that every edge in $H$ is
contained inside an edge of $\hat{G}^{(l)}$. But $\hat{G}^{(l)}$
was chosen in such a way that every $k$-tuple within any edge of
$\hat{G}^{(l)}$ is an edge in $G$. Therefore $G$ contains a copy
of $H$, so we are done.
\end{proof}

As mentioned in the introduction, the proof of Theorem 1
presented here extends in a straightforward manner to the following
multicolour generalisation.

\begin{theorem}
For all positive integers $\Delta$, $k$, and $q$, there exists a
constant $c(\Delta,k,q)$ such that, if $H_1,\ldots,H_q$ are
$k$-uniform hypergraphs each with $n$ vertices and maximum degree
$\Delta$, then
\[r(H_1,\ldots,H_q) \leq c(\Delta, k,q) n.\]
\end{theorem}

The only difference in the proof is in Theorem 4, where we replace $r_k(l)$ by $r_k(l;q)$,
the $q$-colour Ramsey number for the complete $k$-uniform hypergraph on $l$ vertices.
Erd\H{o}s and Rado \cite{ER52} showed that $r_k(l;q) \leq t_k(cl)$, where the constant $c$ depends on $k$ and $q$.
We therefore may take $c(\Delta,k,q) \leq t_k(c\Delta)$ for $k \geq 4$
and $c(\Delta,3,q) \leq t_3(c \Delta \log \Delta)$, where the constant $c$ depends on $k$ and $q$.

\vspace{0.2cm}
\noindent {\bf Remark:} The {\it strong chromatic number} of a hypergraph $H$ is the minimum number of colors required to colour the vertices of
$H$ so that each edge of $H$ has no repeated colour. The proof of Theorem 4 demonstrates that if $H$ is a $k$-uniform hypergraph
with $n$ vertices, maximum degree $\Delta$, and strong chromatic number $l$, then the $q$-colour Ramsey number of $H$ satisfies
\[r(H, \cdots,H) \leq r_k(l;q)^{(2l\Delta)^{l}}.\]
Indeed, in the proof of Theorem 4, we only used the fact that the vertices of $H$ can be partitioned into $l$ parts such that
every edge has at most one vertex in each part.

\section{Lower bound construction}

The following theorem demonstrates that our upper bound for hypergraph Ramsey numbers proved in the previous section
is  in some cases close to best possible.

\begin{theorem}
There is $c>0$ such that for each sufficiently large $\Delta$, there is a $3$-uniform hypergraph $H$
with maximum degree at most $\Delta$ for which the $4$-colour Ramsey number of $H$ satisfies
$$r(H,H,H,H) \geq 2^{2^{c\Delta}}n,$$ where $n$ is the number of vertices of $H$.
\end{theorem}
\begin{proof}
Our proof uses the same $4$-edge-colouring of the complete $3$-uniform hypergraph
that was constructed by Erd\H{o}s and Hajnal (see, e.g., \cite{GRS80}). Not only does this colouring have no large monochromatic complete $3$-uniform
hypergraph, but we show it also does not have any monochromatic copies of a much sparser
$3$-uniform hypergraph $H$.

Let $n \geq 4$ be even, $m = \lceil 2^{n/4} \rceil$, and suppose the edges of the complete graph $K_m$ are coloured red or blue in
such a way that neither colour contains a monochromatic copy of
the graph $K_{n/2}$. Such an edge-colouring exists by the lower bound of Erd\H{o}s (see \cite{GRS80}) on the Ramsey number
of the complete graph.

Let $V = \{v_1, \cdots, v_n\}$ be a set of vertices and let $H$ be
the $3$-uniform hypergraph on $V$ whose edge set is given by
$\{v_i, v_{i+1}, v_j\}$ for all $1 \leq i, j \leq n$. (Note that
when $i = n$, we consider $i+1$ to be equal to 1.) It is straightforward to check that every vertex in $H$
has degree $\Delta \leq 3n$.

We are going to define a $4$-colouring of the complete $3$-uniform
hypergraph on the set
\[T = \{(\gamma_1, \cdots, \gamma_m) : \gamma_i = 0 \mbox{ or }
1\}\] in such a way that there is no monochromatic copy of $H$.
Note that then we will be done, since $T$ has size $2^m \geq
2^{2^{n/4}}$ while $H$ has maximum degree $\Delta \leq 3n$.

To define our colouring, we need some definitions:

If $\epsilon = (\gamma_1, \cdots, \gamma_m)$, $\epsilon' =
(\gamma'_1, \cdots, \gamma'_m)$ and $\epsilon \neq \epsilon'$,
define
\[\delta(\epsilon, \epsilon') = \max\{i : \gamma_i \neq
\gamma'_i\},\] that is, $\delta(\epsilon, \epsilon')$ is the
largest coordinate at which they differ. We can now define an
ordering on $T$ by
\[\epsilon < \epsilon' \mbox{ if } \gamma_i = 0, \gamma'_i = 1,\]
\[\epsilon' < \epsilon \mbox{ if } \gamma_i = 1, \gamma'_i = 0.\]
Another way of looking at this ordering is to assign to each
$\epsilon$ the number $b(\epsilon) = \sum_{i=1}^m \gamma_i
2^{i-1}$. The ordering then says simply that $\epsilon <
\epsilon'$ iff $b(\epsilon) < b(\epsilon')$.

It is important to note the following two properties of the
function $\delta$:\\

(a) if $\epsilon_1 < \epsilon_2 < \epsilon_3$, then
$\delta(\epsilon_1, \epsilon_2) \neq \delta(\epsilon_2,
\epsilon_3)$;

(b) if $\epsilon_1 < \epsilon_2 < \cdots < \epsilon_r$, then
$\delta (\epsilon_1, \epsilon_r) = \max_{1 \leq i \leq r-1}
\delta(\epsilon_i, \epsilon_{i+1})$. \\

Now we are ready to define our colouring of the complete
$3$-uniform hypergraph $\tau$ on vertex set $T$. To begin, suppose
that $\{\epsilon_1, \epsilon_2, \epsilon_3\}$ with $\epsilon_1 <
\epsilon_2 < \epsilon_3$ is an edge in $\tau$. Write $\delta_1 =
\delta(\epsilon_1, \epsilon_2), \delta_2 = \delta(\epsilon_2,
\epsilon_3)$. Then we colour as follows:

\hspace{5.5cm}$C_1$, if $\{\delta_1, \delta_2\}$ is red and $\delta_1 <
\delta_2$;

\hspace{5.5cm}$C_2$, if $\{\delta_1, \delta_2\}$ is red and $\delta_1 >
\delta_2$;

\hspace{5.5cm}$C_3$, if $\{\delta_1, \delta_2\}$ is blue and $\delta_1 <
\delta_2$;

\hspace{5.5cm}$C_4$, if $\{\delta_1, \delta_2\}$ is blue and $\delta_1 >
\delta_2$.
\vspace{0.1cm}

Now, let $S = \{\epsilon_1, \cdots, \epsilon_n\}_{<}$ be an
ordered $n$-tuple within $\tau$ and suppose that there is a copy
of $H$ on $S$ which is coloured by $C_1$. Suppose that the natural
cycle $\{v_1, \cdots, v_n\}$ associated with $H$ occurs as
$\{\epsilon_{\pi(1)}, \cdots, \epsilon_{\pi(n)}\}$ where $\pi$ is
a permutation of $1, \cdots, n$. For each $i$, $1 \leq i \leq n$, let $\phi(i)=\max(\pi(i),\pi(i+1))$ and
$\psi(i)=\min(\pi(i),\pi(i+1))$.

We claim that $\delta_{\phi(i) - 1} =
\delta(\epsilon_{\phi(i) - 1}, \epsilon_{\phi(i)})$ must be larger
than $\delta_j = \delta(\epsilon_j, \epsilon_{j+1})$ for all $j <
\phi(i)-1$.
First consider the triple
$\{\epsilon_{\psi(i)}, \epsilon_{\phi(i) - 1},
\epsilon_{\phi(i)}\}_{<}$, which is an edge of the copy of $H$ on $S$.
The colouring $C_1$ implies that \[\delta_{\phi(i)-1}=\delta(\epsilon_{\phi(i) - 1},
\epsilon_{\phi(i)}) > \delta(\epsilon_{\psi(i)}, \epsilon_{\phi(i) - 1}) =\max_{\psi(i) \leq j < \phi(i)-1} \delta_j.\]
This proves the claim for $\psi(i) \leq j < \phi(i)-1$. Next consider the triple
$\{\epsilon_j,
\epsilon_{\psi(i)}, \epsilon_{\phi(i)}\}_{<}$ with $j<\psi(i)$, which is also an edge of the copy of $H$ on $S$.
The colouring $C_1$ implies that
$$\delta_{j} \leq \delta(\epsilon_j,
\epsilon_{\psi(i)}) < \delta(\epsilon_{\psi(i)}, \epsilon_{\phi(i)})=\delta_{\phi(i)-1}.$$
This proves the claim in the remaining cases $1 \leq j < \phi(i)-1$.

Consider the set $\{\phi(2i-1)\}_{i=1}^{n/2}$, which contains $n/2$ distinct elements
since $\phi(i)=\max(\pi(2i-1),\pi(2i))$ and these pairs are
disjoint. Let $j_1,\ldots,j_{n/2}$ be a permutation of the odd numbers
up to $n-1$ such that $\phi(j_1)<\ldots<\phi(j_{n/2})$. By the claim in the previous paragraph,
we have $\delta_{\phi(j_1) - 1} < \cdots <
\delta_{\phi(j_{n/2})-1}$. Consider, for each $r < s$ with $r, s
\in \{1, \cdots, n/2\}$, the triple $\{\epsilon_{\psi(j_r)},
\epsilon_{\phi(j_r)}, \epsilon_{\phi(j_s)}\}_{<}$, which is an edge of the copy of $H$ on $S$.
Since $\psi(j_r)<\phi(j_r)<\phi(j_s)$, by property (b) of function $\delta$ and the claim above, $\delta(\epsilon_{\psi(j_r)}, \epsilon_{\phi(j_r)})
= \delta_{\phi(j_r) - 1}$ and $\delta(\epsilon_{\phi(j_r)},
\epsilon_{\phi(j_s)}) = \delta_{\phi(j_s) - 1}$. Therefore, by the
definition of $C_1$ we must have that $\{\delta_{\phi(j_r) - 1},
\delta_{\phi(j_s) - 1}\}$ is red. Hence we get a clique of size
$n/2$ in our original colouring. But this cannot happen so we have
a contradiction. All other cases follow similarly, so we're done.
\end{proof}

This result is closely related to another interesting question:
what is the maximum of $r(H)$ over all $k$-uniform hypergraphs
with $m$ edges (we assume here that the hypergraphs we consider do
not have isolated vertices)? For graphs, this question was posed
by Erd\H{o}s and Graham \cite{EG75} who conjectured that the
Ramsey number of a complete graph is at least the Ramsey number of
every graph with the same number of edges. As noted by Erd\H{o}s
\cite{E84}, this conjecture implies that there is a constant $c$
such that for all graphs $G$, $r(G) \leq 2^{c \sqrt{e(G)}}.$ The
best result in this direction, proven by Alon, Krivelevich and
Sudakov \cite{AKS03}, is that $r(G) \leq 2^{c \sqrt{e(G)} \log
e(G)}.$ For hypergraphs, one can naturally ask a question similar
to the Erd\H{o}s-Graham conjecture, i.e, is there a constant $c =
c(k)$ such that for every $k$-uniform hypergraph $H$, $r(H) \leq
t_k(c\sqrt[k]{e(H)})$? The proof of Theorem 6 has the following
corollary:

\begin{corollary}
There is a positive constant $c$ such that for each positive integer $m$, there is a $3$-uniform hypergraph $H$ with at most $m$ edges
 such that the $4$-colour Ramsey number of $H$ satisfies $r(H,H,H,H) \geq 2^{2^{c\sqrt{m}}}$.
\end{corollary}

Indeed the $3$-uniform hypergraph $H$ constructed in the proof of Theorem 6 has $n$ vertices and less than $n^2$ edges,
while $r(H,H,H,H) \geq t_3(n/4) \geq t_3(\sqrt{e(H)}/4)$.  This corollary demonstrates
that the multicolour version of the hypergraph analogue of the Erd\H{o}s-Graham conjecture is false.

In the other direction, we prove the following
theorem:

\begin{theorem}
The $q$-colour Ramsey number of any
$k$-uniform hypergraph $H$ with $m$ edges satisfies
\[r(H, \cdots, H) \leq t_k(c \sqrt{m})\]
for $k \geq 4$, and
\[r(H, \cdots, H) \leq t_3(c \sqrt{m} \log m)\]
for $k=3$, where constant $c$ depends only on $k$ and $q$.
\end{theorem}

Theorem 7 follows immediately from the remark after the proof of Theorem 4 together with the following lemma.
\begin{lemma}
Every $k$-uniform hypergraph $H$ with $m$ edges has strong chromatic number at most $k\sqrt{m}$.
\end{lemma}
\begin{proof}
Let $H'$ be the graph on the same vertex set as $H$ with two vertices adjacent if they lie in an edge of $H$. The strong
chromatic number of $H$ is clearly equal to the chromatic number of $H'$. The number $e(H')$ of edges of $H'$
is at most ${k \choose 2}m \leq {k\sqrt{m} \choose 2}$ since each edge of $H$ gives rise to at most ${k \choose 2}$
edges of $H'$. To finish the proof, note that the chromatic number $\chi$ of any graph with $t$ edges satisfies
${\chi \choose 2} \leq t$ because in an optimal colouring there should be an edge between any two colour classes.
\end{proof}

\section{Conclusion}

Throughout this paper we have aimed for simplicity in the
exposition. Accordingly, in proving Theorem 2, we have cut some
corners to make the proof as pithy as possible. The resulting
constant, $c = (2k)^{k-1}$, is doubtless far from best possible,
but we believe that this loss is outweighed by the resulting
brevity of exposition.

As we noted in the introduction, our Theorem 4 implies that, for
$k \geq 4$, there exists a constant $c = c(k)$ such that, for any
graph $H$ on $n$ vertices with maximum degree $\Delta$, $r(H) \leq
t_k(c \Delta) n$, where the constant $c$ depends only on $k$. For
$k=3$, however, it only implies that
\begin{equation} \label{1} r(H) \leq 2^{2^{c\Delta \log \Delta}} n,\end{equation}
which could perhaps be improved a little. It is
worth noting also that for $k = 2$, the best known bound, proved
by Graham, R\"odl and Ruci\'nski \cite{GRR00} using a very
different method is
\begin{equation} \label{2} r(H) \leq 2^{c \Delta \log^2 \Delta} n.\end{equation}
In light of the situation for higher $k$ as well as the lower bound constructions for $k=2,3$, the following is a
natural question:

\begin{problem}
Can the $\log$ factors in the highest exponent of the upper bounds (\ref{1}) and (\ref{2}) be removed?
\end{problem}

This problem is certainly difficult in the $k=2$ case, but maybe a different extension of the methods of \cite{FS07}
or an appropriate generalisation of the work of Graham, R\"odl and
Ruci\'nski could resolve the $k=3$ case.

It also seems likely to us that the lower bound for this problem
is essentially the same as the upper bound. So we have the following open problem:

\begin{problem}
Is it true that for all $k$ and $\Delta$ and sufficiently large $n$, there exists a $k$-uniform
hypergraph $H$ with maximum degree $\Delta$ and $n$ vertices such
that $r(H) \geq t_k(c\Delta) n$, where $c>0$ only depends on $k$?
\end{problem}

\noindent {\bf Acknowledgement.}\, We would like to thank Jan Hladky
for finding several typos in an earlier version of this paper.

\end{document}